\documentclass{amsart}
\usepackage{amssymb}

\usepackage{tikz}
\usetikzlibrary{cd}
\usetikzlibrary{decorations.markings}

\usepackage{amsmath}
\usepackage{mathscinet}
\usepackage{url}

\newcommand{\bk}{\Bbbk}
\newcommand{\Z}{\mathbb{Z}}
\newcommand{\R}{\mathbb{R}}
\newcommand{\C}{\mathbb{C}}

\newcommand{\simto}{\overset{\sim}{\to}}
\DeclareMathOperator{\Hom}{Hom}
\DeclareMathOperator{\End}{End}
\DeclareMathOperator{\Sym}{Sym}
\newcommand{\id}{\mathrm{id}}
\newcommand{\For}{\mathsf{For}}
\DeclareMathOperator{\Ch}{ch}

\newcommand{\ngmod}{\mathsf{mod}}
\newcommand{\gmod}{\mathsf{gmod}}

\newcommand{\lh}{\text{-}}

\newcommand{\Kb}{K^{\mathrm{b}}}
\newcommand{\Db}{D^{\mathrm{b}}}

\newcommand{\SL}{\mathrm{SL}}
\newcommand{\GL}{\mathrm{GL}}

\newcommand{\fg}{\mathfrak{g}}
\newcommand{\cO}{\mathcal{O}}
\newcommand{\gr}{\mathrm{gr}}

\renewcommand{\H}{\mathrm{H}}
\newcommand{\uw}{\underline{w}}

\newcommand{\slreal}{\fh^{\SL_2}_\bk}
\newcommand{\fh}{\mathfrak{h}}

\newcommand{\cA}{\mathcal{A}}
\newcommand{\cF}{\mathcal{F}}
\newcommand{\cG}{\mathcal{G}}

\newcommand{\cH}{\mathcal{H}}
\newcommand{\ocH}{\overline{\cH}}
\newcommand{\ucH}{\underline{\cH}}
\newcommand{\SBim}{\mathbb{S}\mathrm{Bim}}
\newcommand{\oSBim}{{\overline{\SBim}}}

\newcommand{\inv}{\mathrm{inv}}

\DeclareMathOperator{\uHom}{\underline{Hom}}
\DeclareMathOperator{\uEnd}{\underline{End}}
\newcommand{\LM}{\mathsf{LM}}

\newcommand{\la}{\langle}
\newcommand{\ra}{\rangle}
\newcommand{\ustar}{\mathbin{\underline{\star}}}
\newcommand{\lfrown}{\mathbin{\acute{\frown}}}

\newcommand{\FM}{\mathsf{FM}}
\newcommand{\TiltFM}{\mathsf{TiltFM}}
\newcommand{\hatstar}{\mathbin{\widehat{\star}}}
\newcommand{\Tmon}{\widetilde{T}}
\newcommand{\kappamon}{\kappa^{\mathrm{mon}}}

\newsavebox\BsRbox
\savebox\BsRbox{%
 \begin{tikzpicture}[thick,baseline=-2pt,scale=0.05]
  \draw (0,-5) to (0,0);
  \node at (0,0) {\small $\bullet$};
 \end{tikzpicture}
}
\newcommand{\BsR}{\hskip-3pt\usebox\BsRbox\hskip-5pt} 

\newsavebox\RBsbox
\savebox\RBsbox{%
 \begin{tikzpicture}[thick,baseline=-2pt,scale=0.05]
  \node at (0,0) {\small $\bullet$};
  \draw (0,0) to (0,5);
 \end{tikzpicture}
}
\newcommand{\RBs}{\hskip-3pt\usebox\RBsbox\hskip-5pt}

\newsavebox\BsBsbox
\savebox\BsBsbox{%
 \begin{tikzpicture}[thick,baseline=-2pt,scale=0.05]
  \draw (0,-5) to (0,5);
  \node at (0,0) {}; 
 \end{tikzpicture}
}

\newsavebox\BsRBsbox
\savebox\BsRBsbox{%
 \begin{tikzpicture}[thick,baseline=-2pt,scale=0.05]
  \draw (0,-5) to (0,-1.7);
  \node at (0,-1.7) {\small $\bullet$};
  \node at (0,1.7) {\small $\bullet$};
  \draw (0,1.7) to (0,5);
 \end{tikzpicture}
}
\newcommand{\BsRBs}{\hskip-2pt\usebox\BsRBsbox\hskip-5pt}

\usetikzlibrary{snakes}
\tikzset{t/.style={snake=zigzag,segment amplitude=0.7pt,segment length=1mm}}

\newsavebox\BtRbox
\savebox\BtRbox{%
 \begin{tikzpicture}[thick,baseline=-2pt,scale=0.05]
  \draw[t] (0,-5) to (0,0);
  \node at (0,0) {\small $\bullet$};
 \end{tikzpicture}}
\newcommand{\BtR}{\hskip-3pt\usebox\BtRbox\hskip-2pt} 

\newsavebox\RBtbox
\savebox\RBtbox{%
 \begin{tikzpicture}[thick,baseline=-2pt,scale=0.05]
  \node at (0,0) {\small $\bullet$};
  \draw[t] (0,0) to (0,5);
 \end{tikzpicture}
}
\newcommand{\RBt}{\hskip-3pt\usebox\RBtbox\hskip-5pt}

\newsavebox\BtBtbox
\savebox\BtBtbox{%
 \begin{tikzpicture}[thick,baseline=-2pt,scale=0.05]
  \draw[t] (0,-5) to (0,5);
  \node at (0,0) {}; 
 \end{tikzpicture}
}

\newsavebox\BtRBtbox
\savebox\BtRBtbox{%
 \begin{tikzpicture}[thick,baseline=-2pt,scale=0.05]
  \draw[t] (0,-5) to (0,-1.7);
  \node at (0,-1.7) {\small $\bullet$};
  \node at (0,1.7) {\small $\bullet$};
  \draw[t] (0,1.7) to (0,5);
 \end{tikzpicture}
}

\newsavebox\BsRBtbox
\savebox\BsRBtbox{%
 \begin{tikzpicture}[thick,baseline=-2pt,scale=0.05]
  \draw (0,-6) to (0,-1.7);
  \node at (0,-1.7) {\small $\bullet$};
  \node at (0,1.7) {\small $\bullet$};
  \draw[t] (0,1.7) to (0,6);
 \end{tikzpicture}
}
\newcommand{\BsRBt}{\hskip-2pt\usebox\BsRBtbox\hskip-5pt}

\newcommand{\BsRtensorBtR}{\hskip-3pt\usebox\BsRbox\hskip-9pt\usebox\BtRbox\hskip-2pt}
\newcommand{\RBstensorRBt}{\hskip-3pt\usebox\RBsbox\hskip-9pt\usebox\RBtbox\hskip-5pt}

\numberwithin{equation}{section}
\numberwithin{figure}{section}
\newtheorem{thm}{Theorem}[section]
\newtheorem{lem}[thm]{Lemma}

\theoremstyle{definition}
\newtheorem{defn}[thm]{Definition}

\newtheorem{question}[thm]{Question}
\theoremstyle{remark}
\newtheorem{rmk}[thm]{Remark}
\newtheorem{ex}[thm]{Example}

\title{On monoidal Koszul duality for the Hecke category}
\author{Shotaro Makisumi}
\date{February 25, 2019}
\address{Department of Mathematics, Columbia University, New York, NY 10027, U.S.A.}
\email{makisumi@math.columbia.edu}

\begin{document}

\begin{abstract}
We attempt to give a gentle (though ahistorical) introduction to Koszul duality phenomena for the Hecke category, focusing on the form of this duality studied in joint work \cite{AMRWa, AMRWb} of Achar, Riche, Williamson, and the author. We illustrate some key phenomena and constructions for the simplest nontrivial case of (finite) $\SL_2$ using Soergel bimodules, a concrete algebraic model of the Hecke category.
\end{abstract}

\maketitle

\section{Introduction}\label{s:intro}

Monoidal Koszul duality for the Hecke category categorifies a natural ring involution of the Hecke algebra. Such an equivalence of monoidal categories was originally established in the language of mixed $\ell$-adic sheaves on (Kac--Moody) flag varieties by Bezrukavnikov--Yun \cite{BY}. An important feature of this equivalence is that it involves two rather different categories of sheaves on Langlands dual flag varieties: one side is the more classical Hecke category of Borel-equivariant semisimple complexes, whereas the dual side requires the introduction of what loc.~cit.~calls ``free-monodromic tilting sheaves.''

In recent joint work of Achar, Riche, Williamson, and the author \cite{AMRWa}, we proposed a new construction of the latter category that makes sense for positive characteristic coefficients. This category was used in \cite{AMRWb} to formulate and prove a positive characteristic monoidal Koszul duality for Kac--Moody groups. The latter result, combined with a recent string of advances in modular geometric representation theory (Achar--Rider \cite{ARider}, Mautner--Riche \cite{MR}, Achar--Riche \cite{AR}), yields a character formula for tilting modules of connected reductive groups in characteristic $p$ in terms of $p$-Kazhdan--Lusztig polynomials, confirming (the combinatorial consequence of) the Riche--Williamson conjecture \cite{RW}, for $p$ greater than the Coxeter number.

This article will \emph{not} discuss this or other applications of Koszul duality to representation theory (aside from brief remarks in \S\ref{ss:why-called-koszul}). For that, the reader is referred to the introduction to \cite{AMRWb}, Achar--Riche's survey \cite{AR:formal-koszul}, Williamson's surveys \cite{Wil:Takagi, Wil:ICM}, and Riche's habilitation thesis \cite{Ric:habilitation}.

Instead, our goal is to motivate and explain some of the constructions in \cite{AMRWa}, assuming as little background as possible (some Lie theory and homological algebra). Some key phenomena and constructions can already been seen for the finite flag variety of $\SL_2$, and we illustrate them using an algebraic incarnation of the Hecke category known as Soergel bimodules, which are certain graded bimodules over polynomial rings. However, we will not be able to explain every construction in loc.~cit. In particular, we do not discuss the monoidal structure on free-monodromic complexes (``free-monodromic convolution''), which remains mysterious to this author.

In keeping with this goal, we ignore various technical details, including some that are crucial for the application to modular representation theory. In particular, we mostly ignore the precise characteristic assumptions in \cite{AMRWa, AMRWb}, as well as the bigrading on the various (dgg-)algebras and Hom spaces in \cite{AMRWa}. The reader is directed to the original papers for all precise definitions and statements.

\subsection{Contents}
This article is organized as follows. In \S\ref{s:hecke}, we define the Hecke algebra, explain what ``Hecke category'' means, and define Soergel bimodules. In \S\ref{s:missing-self-duality}, we introduce a ring involution of the Hecke algebra and discuss its categorification (Koszul duality). In particular, \S\ref{ss:why-no-self-duality} explains why the Hecke category cannot be Koszul self-dual. In \S\ref{s:koszul-regular-module}, we discuss a version of Koszul duality for the regular module category of the Hecke category (Theorem~\ref{thm:koszul-regular-rep}), and explain in \S4.4 a key construction from \cite[\S4.4]{AMRWa} (``left-monodromic complexes'' and ``left monodromy action''). In \S\ref{s:mkd}, we state (Theorem~\ref{thm:monoidal-koszul-duality}) the monoidal Koszul duality of \cite{AMRWb}, and give some explanation of the ``free-monodromic complexes'' of \cite[\S5]{AMRWa}.

\subsection{Acknowledgements} It is a pleasure to thank P.~N.~Achar, S.~Riche, and G.~Williamson for the stimulating collaboration. I thank V.~Futorny for inviting me to speak at the XXII Latin American Algebra Colloquium. I also thank G.~Dhillon and U.~Thiel for comments on a draft of this article, and the anonymous referee for a very careful reading and helpful corrections.

Much of this article was written while the author was visiting the Freiburg Institute for Advanced Studies under the program ``Cohomology in Algebraic Geometry and Representation Theory'' organized by A.~Huber-Klawitter, S.~Kebekus, and W.~Soergel. I thank W.~Soergel for the invitation, and FRIAS for the excellent working conditions.


\section{Hecke algebra and Hecke category}\label{s:hecke}

Throughout, $(W,S)$ is a Coxeter system (with $|S| < \infty$). For $s, t \in S$, let $m_{st}$ be the order (possibly $\infty$) of $st$. An important class of examples is that of crystallographic Coxeter systems, which arise as Weyl groups of Kac--Moody groups (for example, finite and affine Weyl groups).

\subsection{Hecke algebra}\label{ss:hecke-alg}
We begin at the decategorified level. Let us recall the Hecke algebra, an algebra over Laurent polynomials $\Z[v,v^{-1}]$, following Soergel's normalization \cite{Soe}.

\begin{defn}
The \emph{Hecke algebra} $\H(W)$ is the (unital associative) $\Z[v,v^{-1}]$-algebra generated by symbols $\delta_s$ for $s \in S$, subject to the following two types of relations:
\begin{equation} \label{eqn:hecke-quad-rel}
 \text{(\emph{quadratic relation})} \ \ (\delta_s+v)(\delta_s-v^{-1}) = 0 \ \ \mbox{ for all } s \in S, \hskip 53pt
\end{equation}
\begin{equation} \label{eqn:hecke-braid-rel}
 \text{(\emph{braid relation})} \ \ \underbrace{\delta_s\delta_t\delta_s \cdots}_{m_{st} \text{ terms}} = \underbrace{\delta_t\delta_s\delta_t \cdots}_{m_{st} \text{ terms}} \ \ \mbox{ for all } s,t \in S \text{ with } m_{st} < \infty.
\end{equation}
\end{defn}
For $w \in W$, set $\delta_w := \delta_{s_1} \cdots \delta_{s_k}$, where $\uw = (s_1, \ldots, s_k)$ is a reduced expression for $w$. By Matsumoto's theorem and \eqref{eqn:hecke-braid-rel}, $\delta_w$ is independent of the choice of $\uw$. It is a classical fact that the elements $\{\delta_w\}_{w \in W}$ form a $\Z[v,v^{-1}]$-basis of $\H(W)$, called the \emph{standard basis}.

The Hecke algebra admits a ring map to the group algebra
\begin{equation} \label{eqn:q}
 \H(W) \to \Z[W] = \bigoplus_{w \in W}\Z e_w: \quad \delta_w \mapsto e_w, \quad v \mapsto 1,
\end{equation}
inducing a ring isomorphism $\Z \otimes_{\Z[v,v^{-1}]} \H(W) \simto \Z[W]$, where $\Z[v,v^{-1}] \to \Z$ sends $v \mapsto 1$. For instance, \eqref{eqn:q} sends \eqref{eqn:hecke-quad-rel} to the relation $e_s^2 = 1$. One says that $\H(W)$ is a deformation of $\Z[W]$, with the natural basis of $\Z[W]$ deforming to the standard basis.

In their seminal work \cite{KL}, Kazhdan and Lusztig used an involution $\overline{(-)}: \H(W) \to \H(W)$ (determined by $\overline{\delta_s} = \delta_s^{-1}$ for $s \in S$ and $\overline{v} = v^{-1}$) to define a new $\Z[v,v^{-1}]$-basis $\{b_w\}_{w \in W}$ of $\H(W)$, nowadays called the \emph{Kazhdan--Lusztig basis}. These elements are determined by the conditions
\[
 \overline{b_w} = b_w, \qquad b_w \in \delta_w + \bigoplus_{x < w}v\Z[v]\delta_x,
\]
where $<$ is the Bruhat order on $W$. In our normalization,
\[
 b_\id = \delta_\id = 1, \qquad b_s = \delta_s + v \text{ for } s \in S.
\]
The Kazhdan--Lusztig basis should be understood via categorifications of $\H(W)$, which we describe next.

\subsection{Categorification of the Hecke algebra}\label{ss:categorification}
The Hecke category is a categorification of the Hecke algebra. Let us explain what this means.

Let $\cA$ be an additive graded monoidal category. In particular, $\cA$ has finite direct sums, and comes equipped with a tensor product $\otimes$ and a ``grading shift'' autoequivalence $(1)$ (compatible with the monoidal structure), whose $n$-th power ($n \in \Z$) will be denoted by $(n)$.

Let $[\cA]_\oplus$ be the split Grothendieck group of the additive category $\cA$. That is, $[\cA]_\oplus$ is the abelian group spanned by isomorphism classes $[B]$ of objects $B \in \cA$, modulo the relations $[B'] + [B''] = [B]$ whenever $B \cong B' \oplus B''$. We make $[\cA]_\oplus$ into a $\Z[v,v^{-1}]$-algebra via $[B] \otimes [B'] = [B \otimes B']$ and $v[B] = [B(1)]$.

For later use, we define the \emph{graded Hom} for such a category by
\begin{equation} \label{eqn:graded-Hom}
 \Hom_\cA^\bullet(X, Y) := \bigoplus_{n \in \Z} \Hom_\cA(X, Y(n))
\end{equation}
for all $X,Y \in \cA$. We also write $\End_\cA^\bullet(X)$ for $\Hom_\cA^\bullet(X,X)$.

The following ad hoc definition illustrates what properties we seek from the Hecke category.
\begin{defn} \label{defn:categorification}
 Let $\cA$ be an additive graded monoidal category. We say that $\cA$ \emph{categorifies} (or is a \emph{categorification of}) $\H(W)$ if it satisfies the following conditions.
\begin{enumerate}
\item It is Krull--Schmidt, with indecomposable objects $B_w$, $w \in W$, and a bijection
\begin{align*}
 \{\text{indecomp.~objects in } \cA \}/\cong &\ \overset{\text{1:1}}{\longleftrightarrow}\  W \times \Z \\
 B_w(n) &\ \longleftrightarrow\  (w,n).
\end{align*}

\item There is an isomorphism of $\Z[v,v^{-1}]$-algebras
\[
 \H(W) \simto [\cA]_\oplus
\]
determined by $b_s \mapsto [B_s]$ for $s \in S$. (Note that $\{b_s\}_{s \in S}$ generate $\H(W)$.) The inverse isomorphism $\Ch: [\cA]_\oplus \simto \H(W)$ is called the \emph{character map}.
\end{enumerate}
\end{defn}
The main consequence of condition (1) is that every object in $\cA$ is isomorphic to a finite direct sum
\[
 \bigoplus_{(w,n) \in W \times \Z} B_w(n)^{\oplus m_{w,n}},
\]
where the multiplicities $m_{w,n} \in \Z_{\ge 0}$ are uniquely determined. It follows that
\[
 [\cA]_\oplus = \bigoplus_{w \in W} \Z[v,v^{-1}] \cdot [B_w],
\]
so that $[\cA]_\oplus$ is isomorphic to $\H(W)$ as a $\Z[v,v^{-1}]$-module. Condition (2) then asks that the multiplications also correspond via a specified map.

\subsection{Hecke category}\label{ss:hecke-cat}
For finite and affine Weyl groups, a geometric categorification of the Hecke algebra has been known since 1980 \cite{KL2}, and has played an extremely important role in geometric representation theory. Here, we discuss generalizations that have appeared since, starting with the work of Soergel \cite{Soe:bimod}.


Unlike the Hecke algebra, which only depends on the Coxeter system, its categorification depends on an additional datum. A \emph{realization} of $(W,S)$ over a field $\bk$ (for simplicity), in the sense of Elias--Williamson \cite[\S3.1]{EW-soergel-calculus}, is a triple
\begin{equation} \label{eqn:realization}
 \fh = (V, \{\alpha_s^\vee\}_{s \in S} \subset V, \{\alpha_s\}_{s \in S} \subset V^*),
\end{equation}
where $V$ is a finite-dimensional $\bk$-vector space and $V^* = \Hom_\bk(V,\bk)$, equipped with \emph{simple roots} $\alpha_s$ and \emph{simple coroots} $\alpha_s^\vee$ indexed by $s \in S$. These elements are required to satisfy certain conditions, most of which are familiar from basic Lie theory. In particular, we ask that $\alpha_s(\alpha_s^\vee) = 2$ for all $s \in S$, and that the assignment $s \mapsto (v \mapsto v - \alpha_s(v)\alpha_s^\vee)$ defines a representation of $W$ on $V$. We assume moreover that the realization is \emph{balanced} and satisfies \emph{Demazure surjectivity}; for precise definitions, see \cite{EW-soergel-calculus} (see also \cite[\S2.1]{AMRWa}).

The reader should ignore these technicalities and instead keep in mind the important class of realizations in Example~\ref{ex:cartan-realizations} below. We mention the general notion of a realization only to emphasize the following two points: the Hecke category can be defined starting from such a combinatorial datum, similar to a root datum but not necessarily arising from a reductive or even Kac--Moody group; and Koszul duality phenomena are expected even in this generality (see Remark~\ref{rmk:non-crystallographic}).
\begin{ex}[Cartan realizations] \label{ex:cartan-realizations}
 Let $G$ be a connected reductive group over $\C$. Choose a Borel subgroup and a maximal torus $B \supset T$. Let $W$ be the Weyl group and $\{\alpha_s\}_{s \in S}$ (resp.~$\{\alpha_s^\vee\}_{s \in S}$) the simple roots (resp.~simple coroots) determined by these choices. Then we obtain a realization of the (finite crystallographic) Coxeter system $(W, S)$ over any field $\bk$ by setting $V := \bk \otimes_\Z X$, where $X$ is the character lattice of $T$. (Then $V^*$ is naturally identified with the base change to $\bk$ of the cocharacter lattice.)
 
 More generally, one may define a realization (over any field $\bk$) starting from a generalized Cartan matrix and an associated Kac--Moody root datum (giving rise to a Kac--Moody group). Realizations arising in this way are called \emph{Cartan realizations}. They are always balanced, and they satisfy Demazure surjectivity possibly with the further assumption that $\bk$ is not of characteristic $2$. For more details, see \cite[\S10.1]{AMRWa} and references therein.
\end{ex}

Let $(\fh,W)$ be a realization over $\bk$. The \emph{Hecke category} is a certain $\bk$-linear additive graded monoidal category $\cH(\fh,W)$ that categorifies $\H(W)$ in the sense of Definition~\ref{defn:categorification}. This category has several different incarnations. In this article, we focus on the most elementary one, the category of \emph{Soergel bimodules}, which is only ``correct'' for some realizations; for a discussion of better-behaved incarnations of the Hecke category, see the discussion following Example~\ref{ex:sbim-sl2} below. However, the more concrete setting of Soergel bimodules will suffice to illustrate some key ideas of Koszul duality.

Let us define Soergel bimodules. Consider the symmetric algebra
\[
 R = \Sym(V^*),
\]
graded with $\deg(V^*) = 2$. In other words, $R$ consists of polynomials in a fixed basis of $V^*$, with double the usual degree. Soergel bimodules form a full subcategory of $R\lh\gmod\lh R$, the category of ($\Z$-)graded $R$-bimodules and graded $R$-bimodule homomorphisms (of degree $0$). Note that $R\lh\gmod\lh R$ has the structures needed to categorify a $\Z[v,v^{-1}]$-algebra: it is additive, monoidal under the tensor product $\otimes_R$, and has a grading shift autoequivalence $(1)$ defined on a graded module $M = \bigoplus_{i \in \Z}M_i$ by $(M(1))_i = M_{i+1}$.

The action of $W$ on $V$ induces an action on $R$. For $s \in S$, consider the graded $R$-bimodule
\[
 B_s := R \otimes_{R^s} R(1),
\]
where $R^s \subset R$ is the $s$-invariants.
\begin{defn} \label{defn:SBim}
 The category $\SBim(\fh,W)$ of \emph{Soergel bimodules} is the smallest full subcategory of $R\lh\gmod\lh R$ that contains $B_s$ for $s \in S$ and is closed under taking finite direct sums $\oplus$, finite tensor products $\otimes_R$, grading shift $(1)$, direct summand $\overset{\oplus}{\subset}$, and under isomorphism. In symbols,
 \[
  \SBim(\fh,W) := \la B_s : s \in S \ra_{\oplus, \otimes_R, (1), \overset{\oplus}{\subset}, \cong} \subset R\lh\gmod\lh R.
 \]
\end{defn}

\begin{ex}[Soergel bimodules for $\SL_2$] \label{ex:sbim-sl2}
Consider the Cartan realization $\slreal$ of $\SL_2$ over a field $\bk$ of characteristic not equal to $2$. That is, $W = S_2 = \{\id, s\}$, $S = \{s\}$, and $V = \bk\alpha_s^\vee$ and $V^* = \bk\alpha_s$, where $\alpha_s(\alpha_s^\vee) = 2$. Then $S_2$ acts on $V$ by $s(\alpha_s^\vee) = -\alpha_s^\vee$. Soergel bimodules are certain graded bimodules over $R = \bk[\alpha_s]$, where $\deg(\alpha_s) = 2$. Since $s(\alpha_s) = -\alpha_s$, we have $R^s = \bk[\alpha_s^2]$.

What are the indecomposable Soergel bimodules up to grading shift and isomorphism? To begin, we have $R$ (the monoidal identity) and $B_s$. By definition, to find new indecomposable bimodules, one should consider direct summands of tensor products (over $R$) of the bimodules we know. Tensor product with $R$ produces nothing new. For $B_s$, using the $R^s$-bimodule decomposition $R = R^s \oplus (R^s\cdot \alpha_s)$, we get
\begin{multline*}
 B_s \otimes_R B_s = R \otimes_{R^s} R \otimes_{R^s} R(2) \cong R \otimes_{R^s} (R^s \oplus R^s(-2)) \otimes_{R^s} R  \\
 \cong R \otimes_{R^s} R \oplus R \otimes_{R^s} R(2) = B_s(-1) \oplus B_s(1).
\end{multline*}
It follows that $B_\id := R$ and $B_s$ are the only indecomposable Soergel bimodules up to grading shift and isomorphism. Moreover, the isomorphism $B_s \otimes_R B_s \cong B_s(-1) \oplus B_s(1)$ decategorifies to $b_s^2 = (v+v^{-1})b_s$, which is equivalent to the quadratic relation \eqref{eqn:hecke-quad-rel}. We therefore have a $\Z[v,v^{-1}]$-algebra isomorphism
\[
 \H(S_2) \overset{\sim}{\longrightarrow} [\SBim(\slreal,S_2)]_\oplus
\]
sending $1 \mapsto [R]$ and $b_s \mapsto [B_s]$, and $\SBim(\slreal,S_2)$ categorifies $\H(S_2)$.
\end{ex}

In \cite{Soe:bimod}, Soergel originally considered his bimodules for realizations satisfying the rather restrictive ``reflection faithfulness'' condition (any Coxeter system admits such a realization over $\R$). Under this assumption, he showed that $\SBim(\fh,W)$ categorifies $\H(W)$.

For general realizations (e.g.~for $W$ an affine Weyl group and $\bk = \overline{\mathbb{F}}_p$), it is expected that Soergel bimodules no longer categorify $\H(W)$. Since Soergel's work, two other incarnations of the Hecke category have appeared, which almost always categorify $\H(W)$:\footnote{There is yet another candidate for the Hecke category: Braden--MacPherson (BMP) sheaves on Bruhat moment graphs, studied by Fiebig \cite{Fie2, Fie1}. Beyond reflection faithful realizations, where this category is equivalent to that of Soergel bimodules, it is not known when BMP sheaves categorify $\H(W)$.}
\begin{enumerate}
 \item \emph{(Geometric)} For Cartan realizations of crystallographic Coxeter groups (Example~\ref{ex:cartan-realizations}), one may consider Borel-equivariant parity complexes (in the sense of Juteau--Mautner--Williamson \cite{JMW}) on the associated Kac--Moody flag variety; see \cite[Part~3]{RW}. This generalizes the original geometric Hecke category, going back to Kazhdan--Lusztig, to positive characteristic coefficients.
 \item \emph{(Diagrammatic)} For realizations of an arbitrary Coxeter system, Elias--Williamson \cite{EW-soergel-calculus} (building on earlier work by Elias--Khovanov \cite{EK} and Elias \cite{Eli}) have defined a diagrammatic Hecke category by generators and relations.\footnote{More precisely, what categorifies the Hecke algebra is the (graded) Karoubi envelope of their diagrammatic category. The Elias--Williamson diagrammatic category is defined even for $\bk$ an integral domain, but its Karoubi envelope does not behave well in this generality.}
\end{enumerate}
For more on these two incarnations of the Hecke category, see Williamson's ICM report \cite{Wil:ICM}.

Let us explain the connection of these categorifications to the Kazhdan--Lusztig basis. In general, a categorification of $\H(W)$ as in Definition~\ref{defn:categorification} yields a $\Z[v,v^{-1}]$-basis $\{\Ch([B_w])\}_{w \in W}$ of $\H(W)$. For specific realizations in characteristic $0$, it is a deep fact (originally proved for Weyl groups by Kazhdan--Lusztig using the Decomposition Theorem in finite field algebraic geometry, and for arbitrary Coxeter systems by Elias--Williamson using a Hodge theory of Soergel bimodules \cite{EW:hodge}) that this categorically defined basis agrees with the combinatorially defined Kazhdan--Lusztig basis. This immediately implies remarkable positivity properties of the latter (positivity of Kazhdan--Lusztig polynomials and structure constants).

Consider a Cartan realization over a field of characteristic $p > 0$. Then the basis defined by the Hecke category is called the \emph{$p$-canonical} (or \emph{$p$-Kazhdan--Lusztig}) \emph{basis} and is denoted by $\{{}^pb_w\}_{w \in W}$ (it only depends on the root datum and $p$). From base change considerations, one sees that
\begin{equation} \label{eqn:pcan-pos}
 {}^pb_w \in \bigoplus_{x \in W}\Z_{\ge 0}[v,v^{-1}]\cdot b_x.
\end{equation}
An emerging new paradigm is that the $p$-canonical basis should control the characteristic $p$ representation theory of Lie-theoretic objects, just as the Kazhdan--Lusztig basis controls their characteristic $0$ representation theory.

\begin{rmk} \label{rmk:non-crystallographic}
Although a large portion of \cite{AMRWa} is written in the language of the Elias--Williamson diagrammatic category, geometry was used to establish several key properties. As a result, the Koszul dualities of \cite{AMRWb} are only proved for Cartan realizations. However, these results are expected to hold for more general realizations, even of non-crystallographic Coxeter systems (but perhaps with stronger characteristic assumptions). For some results in this direction, see \cite{Maka, Makb, ARV}.
\end{rmk}

\section{A missing self-duality}\label{s:missing-self-duality}

\subsection{A ring involution}\label{ss:involution}
At the decategorified level, the main player of this article is a certain ring involution of $\H(W)$ (different from the Kazhdan--Lusztig involution). One sees from the defining relations \eqref{eqn:hecke-quad-rel} and \eqref{eqn:hecke-braid-rel} that there is a ring involution
\[
 \iota: \H(W) \to \H(W)
\]
determined by
\[
 \iota(\delta_s) = \delta_s \quad \text{for } s \in S, \qquad \iota(v) = -v^{-1}.
\]
Applying $\iota$ to the Kazhdan--Lusztig basis yields a basis $\{t_w := \iota(b_w)\}_{w \in W}$, characterized by conditions similar to those for $\{b_w\}_{w \in W}$. For example,
\[
 t_\id = \delta_\id = 1, \qquad t_s = \delta_s - v^{-1} \quad \text{for } s \in S.
\]
\begin{rmk}
In Kazhdan--Lusztig's original paper \cite{KL}, $b_w$ is denoted by $C'_w$, while $t_w$ is denoted by $C_w$.
\end{rmk}

Fix a realization $\fh$ of $(W,S)$. One can ask if $\iota$ can be lifted to $\cH(\fh,W)$.
\begin{question}[Naive Hope 1]
Is there a monoidal autoequivalence of $\cH(\fh,W)$ that categorifies $\iota$? More precisely, is there a monoidal equivalence
\[
 \kappa: \cH(\fh,W) \simto \cH(\fh,W),
\]
such that the isomorphism $\Ch: [\cH(\fh,W)]_\oplus \simto \H(W)$ of Definition~\ref{defn:categorification}(2) identifies the induced ring isomorphism $[\kappa]_\oplus$ with $\iota$?
\end{question}
As stated, this is clearly impossible. The equation $\iota(v) = -v^{-1}$ forces $[\kappa(B_\id(1))]$ $= -v^{-1}$, whereas by Definition~\ref{defn:categorification}(1) and \eqref{eqn:pcan-pos}, the class of every object in $\cH(\fh,W)$ is a $\Z_{\ge 0}[v,v^{-1}]$-linear combination of $b_w$.

A standard way to handle minus signs in categorification is to pass to the bounded homotopy category $\Kb\cH(\fh,W)$.\footnote{Since $\cH(\fh,W)$ is only additive, not abelian, one cannot talk about its derived category.} In the triangulated Grothendieck group $[\Kb\cH(\fh,W)]_\Delta$, the cohomological shift $[1]$ becomes multiplication by $-1$. More precisely, the full embedding of $\cH(\fh,W)$ into $\Kb\cH(\fh,W)$ as complexes supported in cohomological degree $0$ induces a $\Z[v,v^{-1}]$-algebra isomorphism $[\cH(\fh,W)]_\oplus \simto [\Kb\cH(\fh,W)]_\Delta$ with inverse $e([B^\bullet]) = \sum_{i \in \Z} (-1)^i[B^i]$ for any bounded complex $B^\bullet$.
\begin{question}[Naive Hope 2] \label{q:naive2}
Is there a monoidal triangulated autoequivalence of $\Kb\cH(\fh,W)$ that categorifies $\iota$? More precisely, is there a monoidal triangulated equivalence
\[
 \kappa: \Kb\cH(\fh,W) \simto \Kb\cH(\fh,W)
\]
such that the isomorphism $\Ch \circ e : [\Kb\cH(\fh,W)]_\Delta \simto \H(W)$ identifies the induced ring isomorphism $[\kappa]_\Delta$ with $\iota$?
\end{question}
This, too, turns out to be impossible, for a simple reason we now explain.

\subsection{Why the Hecke category cannot be Koszul self-dual}\label{ss:why-no-self-duality}
As in Example~\ref{ex:sbim-sl2}, consider Soergel bimodules for the $\SL_2$ realization $\slreal$. There are graded $R$-bimodule homomorphisms
\begin{equation} \label{eqn:sbim-sl2-dot-maps}
 \BsR: B_s \to R(1): f \otimes g \mapsto fg, \qquad \RBs: R(-1) \to B_s: f \mapsto f \cdot \Delta_s,
\end{equation}
called ``dot morphisms.''\footnote{This terminology and the symbols $\BsR$ and $\RBs$ come from the analogous morphisms in the diagrammatic Hecke category.} Here, $\Delta_s := \frac{\alpha_s}{2} \otimes 1 + 1 \otimes \frac{\alpha_s}{2}$ is the nonzero element of minimal degree, unique up to scalar, satisfying $f\cdot\Delta_s = \Delta_s\cdot f$ for all $f \in R$.

Every graded Hom (see \eqref{eqn:graded-Hom}) is itself naturally a graded $R$-bimodule. It is easy to check that
\begin{equation} \label{eqn:gEnd-R}
 \End^\bullet(R) = R \cdot \id_R,
\end{equation}
\begin{equation} \label{eqn:sbim-sl2-homs}
 \Hom^\bullet(B_s, R(1)) = R \cdot \BsR, \qquad \Hom^\bullet(R(-1), B_s) = R \cdot \RBs.
\end{equation}

We now explain why there cannot be an equivalence $\kappa$ as in Question~\ref{q:naive2}. In the following, we suppress the isomorphism $\Ch \circ e: [\Kb\SBim(\slreal,S_2)]_\Delta \simto \H(S_2)$. Any such $\kappa$ must send the indecomposable bimodule $B_s$ to an indecomposable bounded complex $\kappa(B_s)$ in $\Kb\SBim(\slreal,S_2)$ with class
\[
 \iota(b_s) = t_s = \delta_s - v^{-1} = b_s - v - v^{-1} \in \H(W).
\]
By \eqref{eqn:sbim-sl2-homs}, these conditions force $\kappa(B_s)$ to be isomorphic to the ``complex''
\begin{equation} \label{eqn:Ts}
T_s := \quad
\begin{tikzcd}
R(1) \\
B_s \ar[u, "\BsR"] \\
R(-1) \ar[u, "\RBs"]
\end{tikzcd}
\end{equation}
in degrees $-1$ through $1$, or $T_s[2m]$ for some $m \in \Z$. But $T_s$ is not a complex!
\begin{equation} \label{eqn:mult-comult}
 \BsR \circ \RBs = \alpha_s \cdot \id_R \neq 0.
\end{equation}

For general $(\fh,W)$, for each $s \in S$ there are ``$s$-colored'' dot morphisms
\[
 \BsR: B_s \to B_\id(1), \qquad \RBs: B_\id(-1) \to B_s
\]
in $\cH(\fh,W)$. Analogous to \eqref{eqn:gEnd-R}, we have a natural identification
\[
 \End^\bullet(B_\id) = R,
\]
so that tensor product with $B_\id$ makes every graded Hom into a graded $R$-bimodule. The analogues of \eqref{eqn:gEnd-R}, \eqref{eqn:sbim-sl2-homs}, \eqref{eqn:mult-comult} (with $R(n)$ replaced by $B_\id(n)$) therefore still make sense, and in fact remain true. Then the analogue of $T_s$ \eqref{eqn:Ts} is again not in $\Kb\cH(\fh,W)$, so that $\kappa$ has nowhere to send $B_s$ to. This is the essential obstruction to the naive hope in Question~\ref{q:naive2}.

Let us pretend for a moment that such a $\kappa$ existed. Note that $\Kb\cH(\fh,W)$ has two shifts: in addition to the cohomological shift $[1]$ of the homotopy category, the grading shift $(1)$ of $\cH(\fh,W)$ can be applied term-by-term, giving an endofunctor of $\Kb\cH(\fh,W)$ that we again denote by $(1)$. Since $\kappa$ is triangulated, we have $\kappa \circ [1] \cong [1] \circ \kappa$. However, $\iota(v) = -v^{-1}$ implies that $\kappa \circ (1) \not\cong (1) \circ \kappa$, and instead suggests a natural isomorphism
\begin{equation} \label{eqn:koszul-shifts}
 \kappa \circ (1) \cong [1](-1) \circ \kappa
\end{equation}
(since $(-1)$ and $[1]$ decategorify to $v^{-1}$ and $-1$, respectively), or perhaps $[1]$ replaced by some odd integer power $[m]$. Let us introduce the notation
\[
 \la1\ra := [1](-1)
\]
for the combined shift on the right hand side of \eqref{eqn:koszul-shifts}.

Each of the remaining two sections of this article describes a modification of the duality of Question~\ref{q:naive2} that does exist. Each equivalence satisfies \eqref{eqn:koszul-shifts} and will be called \emph{Koszul duality}.

\section{Koszul duality for the regular representation}\label{s:koszul-regular-module}
In this section, we describe a Koszul duality (Theorem~\ref{thm:koszul-regular-rep}) for a quotient of the Hecke category that categorifies the regular representation of the Hecke algebra.

\subsection{A quotient of the Hecke category} \label{ss:quotients}
The considerations in \S\ref{ss:why-no-self-duality} also suggest a naive fix for the missing duality: simply work in a quotient of the Hecke category where $\alpha_s \cdot \id_{B_\id} = 0$ for all $s \in S$. This turns out to almost work.

\begin{defn}
 The \emph{left quotient} $\ocH(\fh,W)$ is the category with the same objects as $\cH(\fh,W)$, but whose graded Homs are given by
 \[
  \Hom_{\ocH(\fh,W)}^\bullet(X,Y) := \bk \otimes_R \Hom_{\cH(\fh,W)}^\bullet(X, Y).
 \]
 Here, $\bk$ is viewed as a graded $R$-module via the counit $\epsilon_R: R \to \bk$ (sending $V^*$ to $0$).
\end{defn}
In other words, there is a natural quotient functor
\begin{equation} \label{eqn:cH-to-ocH}
 \For: \cH(\fh,W) \to \ocH(\fh,W)
\end{equation}
that is the identity on objects, and whose induced maps on graded Homs are surjective with kernel spanned by morphisms of the form $\lambda \cdot f$, where $\lambda \in R$ is a polynomial with constant term $0$. It is known that each $B_w$ remains indecomposable under $\For$, so $\For$ induces a $\Z[v,v^{-1}]$-module isomorphism $[\cH(\fh,W)]_\oplus \simto [\ocH(\fh,W)]_\oplus$. While $\ocH(\fh,W)$ is not monoidal, it is a module category for $\cH(\fh,W)$ acting on the right. In fact, it categorifies the right regular module of $\H(W)$ (i.e.~$\H(W)$ acting on itself by right multiplication).

\begin{rmk}
 In the geometric Hecke category (see the discussion after Example~\ref{ex:sbim-sl2}), the quotient functor \eqref{eqn:cH-to-ocH} corresponds to the forgetful functor from Borel-equivariant to Borel-constructible parity complexes.
\end{rmk}

We similarly define the \emph{right quotient} $\ucH(\fh,W)$ via $(-) \otimes_R \bk$, categorifying the left regular representation. There is an ``inversion'' equivalence
\begin{equation} \label{eqn:inversion}
 \inv: \ocH(\fh,W) \simto \ucH(\fh,W)
\end{equation}
intertwining the left and right $R$-actions, and sending $\inv(B_w) \cong B_{w^{-1}}$ for all $w \in W$.

The non-complex $T_s$ of \eqref{eqn:Ts} becomes an actual object in each quotient category. This is an example of an indecomposable ``tilting complex.'' More generally, just as $\Kb\ocH(\fh,W)$ contains the full additive subcategory $\ocH(\fh,W)$ stable under the grading shift $(1)$, it also contains a full additive subcategory of tilting complexes\footnote{Very briefly, the triangulated category $\Kb\ocH(\fh,W)$ admits a ``perverse'' t-structure, whose heart (``mixed perverse complexes'') is a graded highest weight category with shift $\la1\ra$ (see \cite{AR:mixed-derived, Maka, ARV}). For such a category, there is a notion of tilting objects and a classification theorem for the indecomposable tilting objects.} that is stable under the Koszul dual shift $\la1\ra = [1](-1)$. There are indecomposable tilting complexes $T_w$, $w \in W$, and a bijection
\begin{align*}
 \{\text{indecomp.~tilting complexes} \}/\cong &\ \overset{\text{1:1}}{\longleftrightarrow}\  W \times \Z \\
 T_w\la n\ra &\ \longleftrightarrow\  (w,n),
\end{align*}
Koszul dual to the bijection in Definition~\ref{defn:categorification}(1), or rather its analogue for $\ocH(\fh,W)$.

These quotient categories are still not quite self-dual. Given a realization $\fh$ as in \eqref{eqn:realization}, the \emph{dual realization}
\[
 \fh^* := (V^*, \{\alpha_s\}_{s \in S} \subset V^*, \{\alpha_s^\vee\}_{s \in S} \subset V)
\]
is obtained by exchanging $V$ with $V^*$ and simple roots with simple coroots. Let $R^\vee = \Sym(V)$, again graded with $\deg V = 2$. For Cartan realizations of a reductive or Kac--Moody group, the dual realization is associated to the Langlands dual group.

\subsection{Statement}\label{ss:self-duality}
Our first Koszul duality relates the left and right regular module categories for dual realizations.
\begin{thm}[Koszul duality for Kac--Moody groups \cite{AMRWb}] \label{thm:koszul-regular-rep}
 Let $(\fh,W)$ be a Cartan realization over a field $\bk$ of characteristic not equal to $2$. There is a triangulated equivalence
 \[
  \kappa: \Kb\ocH(\fh, W) \simto \Kb\ucH(\fh^*, W)
 \]
 satisfying $\kappa \circ (1) \cong \la1\ra \circ \kappa$. Moreover, $\kappa(B_w) \cong T_w$ and $\kappa(T_w) \cong B_w$, and $\kappa$ categorifies $\iota$.
\end{thm}
When $B_w$ categorifies the Kazhdan--Lusztig basis $b_w$, Theorem~\ref{thm:koszul-regular-rep} implies that $T_w$ categorifies the dual basis $t_w = \iota(b_w)$. In general, $\{B_w\}$ and $\{T_w\}$ give two bases of the Hecke algebra that are exchanged by $\iota$.

In \S\ref{ss:LM}, we discuss a key ingredient towards the proof of Theorem~\ref{thm:koszul-regular-rep}.

\begin{ex}
Since the $\SL_2$ realization $\slreal$ (see Example~\ref{ex:sbim-sl2}) is self-dual, Theorem~\ref{thm:koszul-regular-rep} combined with inversion \eqref{eqn:inversion} gives a triangulated autoequivalence
\[
 \kappa': \Kb\oSBim(\slreal,S_2) \simto \Kb\oSBim(\slreal,S_2)
\]
satisfying $\kappa' \circ (1) \cong \la1\ra \circ \kappa'$, and sending $\kappa'(B_s) \cong T_s$ and $\kappa'(T_s) \cong B_s$.
\end{ex}

\subsection{Why ``Koszul''?}\label{ss:why-called-koszul}
Classically, Koszul duality relates seemingly unrelated graded rings $A$ and $A^!$, usually via a derived equivalence of their graded module categories $A\lh\gmod$ and $A^!\lh\gmod$ (with appropriate finiteness conditions). The most classical example is due to Bernstein--Gelfand--Gelfand \cite{BGG} and goes as follows. Let $V$ be a finite-dimensional $\bk$-vector space, let $V^* = \Hom_\bk(V, \bk)$ be its dual, and consider the symmetric and exterior algebras $R = \Sym(V^*)$ and $\Lambda = \Lambda(V)$, graded with $\deg V^* = 1$ and $\deg V = -1$. Then there exists a triangulated equivalence
\[
 \Db(R\lh\gmod) \supset \la \bk \ra_\Delta \xrightarrow[\sim]{\kappa^{\mathrm{BGG}}} \la \Lambda \ra_\Delta \subset \Db(\Lambda\lh\gmod)
\]
satisfying
\begin{equation} \label{eqn:koszul-shifts-2}
 \kappa^{\mathrm{BGG}} \circ \la 1 \ra \cong \la-1\ra[1] \circ \kappa^{\mathrm{BGG}},
\end{equation}
where $\la1\ra$ is the endofunctor that shifts the module grading down by $1$ (applied term-by-term to a complex). In particular, \eqref{eqn:koszul-shifts-2} means that one cannot forget the gradings to obtain a functor relating their ungraded module categories $\Db(R\lh\ngmod)$ and $\Db(\Lambda\lh\ngmod)$; it is the gradings that reveal the hidden relation between $R$ and $\Lambda$. The name ``Koszul duality'' comes from the Koszul complex \eqref{eqn:koszul-resolution}, which plays a key role in this derived equivalence.

Koszul duality as a phenomenon in Lie theory goes back to Beilinson--Ginzburg--Soergel \cite{BGS}. Let $\fg$ be a complex semisimple Lie algebra. Choose a Borel and a Cartan subalgebra $\mathfrak{b} \supset \fh$, and consider the BGG category $\cO = \cO(\fg,\mathfrak{b},\fh)$. As explained in loc.~cit., one can define a ``graded version'' $\cO_0^\gr$ of its principal block $\cO_0$, which comes with a ``grading shift'' autoequivalence, a forgetful functor, and a natural isomorphism
\[
 \la1\ra: \cO_0^\gr \to \cO_0^\gr, \qquad \For: \cO_0^\gr \to \cO_0, \qquad \For \circ \la1\ra \cong \For,
\]
analogous to the forget-the-grading functor $A\lh\gmod \to A\lh\ngmod$. The Koszul self-duality of \cite{BGS} is a triangulated autoequivalence
\[
 \kappa^{\mathrm{BGS}}: \Db\cO_0^\gr \simto \Db\cO_0^\gr,
\]
intertwining the shifts as in \eqref{eqn:koszul-shifts-2}. Thus Koszul duality is a hidden self-duality of $\cO_0$ revealed by its graded version $\cO_0^\gr$.

In the setting of the previous paragraph, one obtains a realization $\fh^{\mathfrak{g}}$ over $\C$ of the Weyl group $W$ by setting $V = \fh$ with the usual notion of simple roots and coroots. A result of Achar--Riche \cite{AR:frobenius} identifies $\Kb\cH(\fh^{\mathfrak{g}},W)$ with $\Db\cO_0^\gr$, and Theorem~\ref{thm:koszul-regular-rep} recovers the Beilinson--Ginzburg--Soergel result.\footnote{More precisely, Theorem~\ref{thm:koszul-regular-rep} becomes $\kappa^{\mathrm{BGS}}$ composed with the Ringel self-duality of $\cO_0^\gr$.} By specializing to some other realizations arising naturally in Lie theory, Theorem~\ref{thm:koszul-regular-rep} and its variants yield further derived equivalences of appropriate graded versions of certain categories of representations.

We should comment that the classical Koszul duality deals with what are called Koszul graded rings. In the setting of this article, the algebras involved are in general not Koszul, and what remains are the homological patterns such as \eqref{eqn:koszul-shifts-2}. We continue to call these equivalences Koszul duality.

It seems to this author that Koszul duality in this general sense is the more basic phenomenon in Lie theory, whereas Koszulity lies deeper, being related to the question of exceptional primes (i.e.~when modular representation theory does not behave like characteristic $0$ representation theory). For instance, the Koszulity of $\cO_0^\gr$ is essentially equivalent to the Kazhdan--Lusztig conjecture.

\subsection{Monodromy action and left-monodromic complexes}\label{ss:LM}

In the generality stated in Theorem~\ref{thm:koszul-regular-rep}, Koszul duality is obtained as a consequence of the monoidal Koszul duality (Theorem~\ref{thm:monoidal-koszul-duality}) discussed in \S\ref{s:mkd} below. Nevertheless, we can already explain a key ingredient towards Theorem~\ref{thm:koszul-regular-rep}.

Let $f \in R$ be homogeneous of degree $d$. Since $\ocH(\fh,W)$ is a left quotient, right multiplication by $f$ defines a morphism $m(f)_B : B \to B(d)$ for any $B \in \ocH(\fh,W)$, and more generally
\[
 m(f)_\cF : \cF \to \cF(d)
\]
for any complex $\cF \in \Kb\ocH(\fh,W)$. These morphisms define a natural transformation $m(f) : \id \to (d)$ of endofunctors of $\Kb\ocH(\fh,W)$, and varying $f \in R$, we obtain a graded algebra map
\[
 m : R \to \bigoplus_{d \in \Z} \Hom(\id, (d)).
\]
Thus $R$ acts functorially on every object in $\Kb\ocH(\fh,W)$ in a graded way compatible with $(1)$. For the purpose of these paragraphs, we say that $R$ acts on the category with shift $(\Kb\ocH(\fh,W), (1))$.

Theorem~\ref{thm:koszul-regular-rep} implies an additional action. Since $R^\vee = \Sym(V)$ similarly acts by left multiplication on $(\Kb\ucH(\fh^*,W), (1))$, via $\kappa$ it should also act on $(\Kb\ocH(\fh,W), \la1\ra)$. That is, we expect a graded algebra map
\begin{equation} \label{eqn:monodromy}
 \mu : R^\vee \to \bigoplus_{d \in \Z} \Hom(\id, \la d \ra).
\end{equation}
Concretely, for $h \in R^\vee$ homogeneous of degree $d$, we expect morphisms
\[
 \mu(h)_\cF : \cF \to \cF\la d \ra
\]
functorial in $\cF \in \Kb\ocH(\fh,W)$.

One key construction in the paper \cite{AMRWa} is to find this hidden action \eqref{eqn:monodromy}, called the \emph{left monodromy action}. This is done by replacing $\Kb\ocH(\fh,W)$ with an equivalent triangulated category $\LM(\fh,W)$ of \emph{left-monodromic complexes}, where this action becomes more visible.

We begin by describing $\Kb\ocH(\fh,W)$ in a slightly different way. Given a complex $\cF \in \Kb\ocH(\fh,W)$, first consider the underlying \emph{$\cH(\fh,W)$-sequence} $\cF = (\cF^i)_{i \in \Z}$, i.e.~a $\Z$-graded sequence of objects $\cF^i \in \cH(\fh,W)$. Its graded endomorphism ring
\[
 \uEnd(\cF) := \prod_{p,q,d} \Hom_{\cH(\fh,W)}(\cF^p,\cF^q(d)),
\]
is an $R$-bimodule, bigraded by homological degree and the grading in $\cH(\fh,W)$. The counit $\epsilon_R : R \to \bk$ induces a map
\[
 \epsilon_R : \uEnd(\cF) \to \bk \otimes_R \uEnd(\cF).
\]
Then an object of $\Kb\ocH(\fh,W)$ may be viewed as a pair $(\cF,\delta)$, where $\cF$ is the underlying sequence as above, and $\delta$ is an element of $\uEnd(\cF)$ of an appropriate bidegree (that we will not specify) satisfying $\epsilon_R(\delta \circ \delta) = 0$.

Roughly speaking, $\Kb\ocH(\fh,W)$ is obtained from $\Kb\cH(\fh,W)$ by applying $\bk \otimes_R (-)$. The idea behind left-monodromic complexes is to replace $\bk$ by its Koszul resolution, a resolution as a graded $R$-module:
\[
\begin{tikzcd}
 \cdots \ar[r] & (V^* \wedge V^*) \otimes R \ar[r] & V^* \otimes R \ar[r] & R \ar[r, "\epsilon_R"] & \bk.
\end{tikzcd}
\]
Here and in the rest of this article, undecorated tensor products are over $\bk$. Let $\Lambda = \Lambda(V^*)$, the exterior algebra of $V^*$ (with an appropriate bigrading, placing $V^*$ in homological degree $-1$). Then the Koszul resolution can be written as
\begin{equation} \label{eqn:koszul-resolution}
 A := \Lambda \otimes R
\end{equation}
equipped with an appropriate differential $\kappa$. Now, we define $\LM(\fh,W)$ by applying $A \otimes_R (-)$ to $\Kb\cH(\fh,W)$, in the following sense. Given two $\cH(\fh,W)$-sequences $\cF$ and $\cG$, define
\[
 \uHom_\LM(\cF, \cG) := \Lambda \otimes \uHom(\cF, \cG) \quad (= A \otimes_R \uHom(\cF, \cG)),
\]
which has a differential $\kappa$ defined via its action on $A$. We also write $\uEnd_\LM(\cF)$ for $\uHom_\LM(\cF,\cF)$.

\begin{defn}[{\cite[Definition~4.4.2]{AMRWa}}] \label{defn:LM}
 An object of $\LM(\fh,W)$, called a \emph{left-monodromic complex}, is a pair $(\cF, \delta)$, where $\cF$ is a $\cH(\fh,W)$-sequence and $\delta \in \uEnd_\LM(\cF)$. We require that $\delta$ is of an appropriate bidegree, and that it satisfies
\begin{equation} \label{eqn:LM-defn}
 \delta \circ \delta + \kappa(\delta) = 0.
\end{equation}
The definition of morphisms in $\LM(\fh,W)$ parallels that for the homotopy category. Given two left-monodromic complexes $(\cF,\delta_\cF)$ and $(\cG,\delta_\cG)$, one makes $\uHom_\LM(\cF,\cG)$ into a complex under the differential
\begin{equation} \label{eqn:uHom-LM-diff}
 d_{\uHom_\LM}(f) = \delta_\cG \circ f - (-1)^{|f|} f \circ \delta_\cF + \kappa(f),
\end{equation}
where $|f|$ is the cohomological (first) degree of $f$. Then $\Hom_{\LM(\fh,W)}(\cF,\cG)$ is defined to be the bidegree $(0,0)$ homology of $\uHom_\LM(\cF,\cG)$. We adopt the usual terminology of dg categories: a \emph{chain map} is a degree $(0,0)$ element $f \in \uHom_\LM(\cF,\cG)$ satisfying $d_{\uHom_\LM}(f) = 0$, and a chain map is \emph{nullhomotopic} if it is of the form $d_{\uHom_\LM}(h)$ for some $h \in \uHom_\LM(\cF,\cG)$. Thus morphisms in $\LM(\fh,W)$ are chain maps modulo homotopy.
\end{defn}
With an appropriate triangulated structure on $\LM(\fh,W)$, the quasi-isomor-\\phism $\epsilon_A : A \simto \bk$ induces a triangulated equivalence
\[
 \For: \LM(\fh,W) \simto \Kb\ocH(\fh,W): (\cF, \delta) \mapsto (\cF, \epsilon_A(\delta)).
\]
Roughly, this functor takes left-monodromic ``differentials'' and chain maps and discards any component that involves a nontrivial $\Lambda$ part.

There is also a natural functor
\[
 \For: \Kb\cH(\fh,W) \to \LM(\fh,W): (\cF, \delta) \mapsto (\cF, \eta_\Lambda(\delta)),
\]
where $\eta_\Lambda: \uHom(\cF,\cG) \to \uHom_\LM(\cF,\cG)$ denotes the maps induced by the unit $\eta_\Lambda: \bk \to \Lambda$. In particular, we have the left-monodromic complex
\begin{equation} \label{eqn:Tid}
 T_\id := \For(B_\id),
\end{equation}
with underlying $\cH(\fh,W)$-sequence $(\ldots, 0, B_\id, 0, \ldots)$, where $B_\id$ is in position $0$, and $\delta_{T_\id} = 0$. However, $\LM(\fh,W)$ contains many objects that do not come from $\Kb\cH(\fh,W)$ in this way.

Let us illustrate these definitions with examples, using Soergel bimodules.
\begin{ex} \label{ex:LM-sl2}
Consider the $\SL_2$ realization $\slreal$ (see Example~\ref{ex:sbim-sl2}). Then
\[
 \For: \LM(\slreal,S_2) \simto \Kb\oSBim(\slreal,S_2)
\]
sends
\begin{equation} \label{eqn:LM-RE-ex1}
\begin{array}{c}\begin{tikzcd}
R(1) \\
B_s \ar[u, "\BsR"] \\
R(-1) \ar[u, "\RBs"] \ar[uu, bend left=60, "-\alpha_s \otimes \id"]
\end{tikzcd}\end{array}
\quad \longmapsto \quad
\begin{array}{c}\begin{tikzcd}
R(1) \\
B_s \ar[u, "\BsR"] \\
R(-1). \ar[u, "\RBs"]
\end{tikzcd}\end{array}
\end{equation}
On the right hand side of \eqref{eqn:LM-RE-ex1} is the complex $T_s$ from \eqref{eqn:Ts}. We saw that $\delta_{T_s} \circ \delta_{T_s}$ has one nonzero component $\alpha_s \cdot \id: R(-1) \leadsto R(1)$. Since this component is killed by $\epsilon_R$, it can be lifted to the $V^* \otimes R$ term of the Koszul resolution:

\[
\begin{tikzcd}[row sep=0]
\alpha_s \otimes R \ar[r] & R \ar[r, "\epsilon_R"] & \bk, \\
\alpha_s \otimes 1 \ar[r, mapsto] & \alpha_s \ar[r, mapsto] & 0.
\end{tikzcd}
\]
In the left-monodromic lift of $T_s$, depicted on the left hand side of \eqref{eqn:LM-RE-ex1}, a new component $-\alpha_s \otimes \id$ of ``chain degree'' $2$, where $\alpha_s$ now lies in the exterior algebra $\Lambda$, records (minus) this lift.
\end{ex}
\begin{ex} \label{ex:LM-sl3}
Consider the $\SL_3$ Cartan realization $\fh^{\SL_3}_\bk$ of $W = S_3$. Let $S = \{s, t\}$, so that $V^* = \bk\alpha_s \oplus \bk\alpha_t$. There are now $s$- and $t$-colored dot morphisms
\begin{align*}
 \BsR: B_s \to R(1), \quad \RBs: R(-1) \to B_s, \quad \BtR: B_t \to R(1), \quad \RBt: R(-1) \to B_t,
\end{align*}
and we set $\BsRBt := \RBt \circ \BsR$. Then
\[
 \For: \LM(\fh^{\SL_3}_\bk,S_3) \simto \Kb\oSBim(\fh^{\SL_3}_\bk,S_3)
\]
sends
\[
\cF :=
\begin{array}{c}\begin{tikzcd}
R(3) \\
B_t(2) \ar[u, "\BtR"] \\
B_s \ar[u] \ar[u, "\BsRBt"] \ar[uu, bend left=60, "-\alpha_t \otimes \BsR" pos=0.6] \\
R(-1) \ar[u, "\RBs"] \ar[uu, bend left=60, "-\alpha_s \otimes \RBt"] \ar[uuu, bend right=60, "(\alpha_s \wedge \alpha_t) \otimes \id"']
\end{tikzcd}\end{array}
\quad \longmapsto \quad
\begin{array}{c}\begin{tikzcd}
R(3) \\
B_t(2) \ar[u, "\BtR"] \\
B_s \ar[u, "\BsRBt"] \\
R(-1) \ar[u, "\RBs"]
\end{tikzcd}\end{array} = \For(\cF).
\]
As in Example~\ref{ex:LM-sl2}, the chain degree $2$ components of $\delta_\cF$ correspond to (a particular choice of) lifts of the components of $\delta_{\For(\cF)} \circ \delta_{\For(\cF)}$
\[
 R(-1) \leadsto B_t(2): \alpha_s \cdot \RBt, \qquad B_s \leadsto R(3): \alpha_t \cdot \BsR
\]
to the $V^* \otimes R$ term of the Koszul resolution. In addition, $\delta_\cF \circ \delta_\cF$ has a component\footnote{The first composition in this display receives an extra minus sign from the Koszul sign rule, once one treats the gradings more carefully.}
\begin{multline*}
 R(-1) \leadsto R(3): \BtR \circ (-\alpha_s \otimes \RBt) + (-\alpha_t \otimes \BsR) \circ \RBs \\
 = \alpha_s \otimes (\alpha_t \cdot \id) - \alpha_t \otimes (\alpha_s \cdot \id),
\end{multline*}
which needs to be lifted to the $(V^* \wedge V^*) \otimes R$ term:
\[
\begin{tikzcd}[row sep=0]
(\alpha_s \wedge \alpha_t) \otimes R \ar[r] & (\alpha_s \otimes R) \oplus (\alpha_t \otimes R) \ar[r] & R \ar[r, "\epsilon_R"] & \bk \\
(\alpha_s \wedge \alpha_t) \otimes 1 \ar[r, mapsto] & \alpha_t \otimes \alpha_s - \alpha_s \otimes \alpha_t \\
\end{tikzcd}
\]
The chain degree $3$ component $(\alpha_s \wedge \alpha_t) \otimes \id$ of $\delta_\cF$ corresponds to this now unique lift, which encodes the choice of lifts made in the chain degree $2$ components.
\end{ex}
\begin{ex}
Again for the $\SL_3$ Cartan realization, there is a morphism of left-monodromic complexes
\[
\begin{tikzcd}[column sep=100pt]
& R(3) \\
B_s \ar[r, "\BsRBt" description] \ar[ru, "-\alpha_t \otimes \BsR"] & B_t(2) \ar[u, "\BtR"] \\
R(-1) \ar[u, "-\RBs"] \ar[ruu, out=140, in=170, distance=80, "(\alpha_s \wedge \alpha_t) \otimes \id"] \ar[ru, "-\alpha_s \otimes \RBt"']
\end{tikzcd}
\]
from the complex in the left hand column to the one in the right hand column (both coming from complexes in $\Kb\SBim(\fh^{\SL_3}_\bk,S_3)$), given by components $\BsRBt + (-\alpha_t \otimes \BsR) + (-\alpha_s \otimes \RBt) + ((\alpha_s \wedge \alpha_t) \otimes \id)$. The three diagonal components again use the Koszul resolution to record the failure of the ``classical'' component $\BsRBt$ to be a genuine map of complexes in $\Kb\SBim(\fh^{\SL_3}_\bk,S_3)$. The left-monodromic complex $\cF$ of Example~\ref{ex:LM-sl3} is the cone of this morphism.
\end{ex}

The point of replacing $\Kb\ocH(\fh,W)$ with the equivalent category $\LM(\fh,W)$ is the following. Consider the derivation $(-) \lfrown (-): V \otimes \Lambda \to \Lambda$ induced by the natural pairing between $V$ and $V^*$:
\[
 x \lfrown (r_1 \wedge \cdots \wedge r_k) = \sum_{i = 1}^k (-1)^{i+1}(r_1 \wedge \cdots \wedge \widehat{r_i} \wedge \cdots \wedge r_k)r_i(x)
\]
for $x \in V$ and $r_1, \ldots, r_k \in V^*$. This induces a map
\[
 (-) \lfrown (-): V \otimes \uEnd_\LM(\cF) \to \uEnd_\LM(\cF).
\]
Now, given $(\cF, \delta) \in \LM(\fh,W)$ and $x \in V$, one can show that $x \lfrown \delta$ defines a morphism $\mu(x)_\cF: \cF \to \cF\la2\ra$ functorial in $\cF$. In particular, the morphisms $\mu(x)_\cF$ and $\mu(y)_ \cF$ commute for $x, y \in V$, and we obtain by composition the desired monodromy action \eqref{eqn:monodromy}.

\begin{ex} For the left-monodromic complex $T_s$ from Example~\ref{ex:LM-sl2}, the morphism $\mu(x)_{T_s}: T_s \to T_s\la2\ra$ is given by the chain map
\[
\begin{tikzcd}[column sep=100pt]
R(1)
& R(-1) \\
B_s \ar[u, "\BsR"]
& B_s(-2) \ar[u] \\
R(-1) \ar[u, "\RBs"] \ar[uu, bend left=60, "-\alpha_s \otimes \id"] \ar[uur, "-\alpha_s(x)"]
& R(-3). \ar[u] \ar[uu, bend left=60]
\end{tikzcd}
\]
(We have omitted the component labels on $T_s\la2\ra$.)

For the left-monodromic complex $\cF$ from Example~\ref{ex:LM-sl3}, the morphism \\$\mu(x)_\cF: \cF \to \cF\la2\ra$ is given by the chain map
\[
\begin{tikzcd}[column sep=100pt]
R(3) 
& R(1) \\
B_t(2) \ar[u, "\BtR"]
& B_t \ar[u] \\
B_s \ar[u] \ar[u, "\BsRBt"] \ar[uu, bend left=60, "-\alpha_t \otimes \BsR" pos=0.6] \ar[uur, "-\alpha_t(x)\BsR" pos=0.65]
& B_s(-2) \ar[u] \ar[u] \ar[uu, bend right=60] \\
R(-1) \ar[u, "\RBs"] \ar[uu, bend left=60, "-\alpha_s \otimes \RBt"] \ar[uuu, bend left=90, distance=90, "(\alpha_s \wedge \alpha_t) \otimes \id_R" description] \ar[uur, "-\alpha_s(x)\RBt"' pos=0.35] \ar[uuur, "(\ast)" description]
& R(-3), \ar[u] \ar[uu, bend right=60] \ar[uuu, bend right=80, distance=50] \\
\end{tikzcd}
\vspace{-20pt}
\]
where $(\ast) = \alpha_t \otimes \alpha_s(x) \cdot \id_R - \alpha_s \otimes \alpha_t(x) \cdot \id_R$, and we have omitted the component labels on $\cF\la2\ra$.

In either case, observe that each ``classical'' component of the monodromy morphism comes from pairing $x$ with a chain degree $2$ component of the left-monodromic differential.
\end{ex}

Since the extra components in a left-monodromic ``differential'' record the failure of complexes in $\Kb\ocH(\fh,W)$ to be a genuine complex in $\Kb\cH(\fh,W)$, one can heuristically think of the left monodromy action as detecting whether complexes in $\Kb\ocH(\fh,W)$ admit a lift to $\Kb\cH(\fh,W)$. It should be emphasized that this is merely a heuristic; while complexes that admit a lift to $\Kb\cH(\fh,W)$ certainly have trivial left monodromy, the reverse implication is false, as shown by the following example.
\begin{ex} \label{ex:trivial-monodromy-lift-counterexample}
For the $\SL_3$ Cartan realization, consider the left-monodromic complex
\[
\cF :=
\begin{array}{c}\begin{tikzcd}
R(2) \\
B_s B_t \ar[u, "\BsRtensorBtR"] \\
R(-2). \ar[u, "\RBstensorRBt"] \ar[uu, bend left=60, "-\alpha_s \otimes (\alpha_t \cdot \id)"]
\end{tikzcd}\end{array}
\]
(We have omitted the monoidal product in $\SBim(\fh^{\SL_3}_\bk,S_3)$ from the notation. For example, $\RBstensorRBt = \RBs \otimes_R \RBt: R(-1) \otimes_R R(-1) \to B_s \otimes_R B_t$.) It is easily seen that $\cF$ (and its image in $\Kb\oSBim(\fh^{\SL_3}_\bk,S_3)$) does not admit a lift to $\Kb\SBim(\fh^{\SL_3}_\bk,S_3)$. However, $\cF$ has trivial left monodromy. Indeed, the monodromy map $\mu(x)_\cF: \cF \to \cF\la2\ra$ is given by the chain map
\[
\begin{tikzcd}[column sep=100pt]
R(2)
& R \\
B_sB_t \ar[u, "\BsRtensorBtR"]
& B_sB_t(-2) \ar[u] \\
R(-2) \ar[u, "\RBstensorRBt"] \ar[uu, bend left=60, "-\alpha_s \otimes (\alpha_t \cdot \id)"] \ar[uur, "-\alpha_s(x)\alpha_t \cdot \id"]
& R(-4), \ar[u] \ar[uu, bend left=60]
\end{tikzcd}
\]
which is nullhomotopic with homotopy given by a single component $R(-2) \leadsto R: -\alpha_t \otimes (\alpha_s(x) \cdot \id)$.
\end{ex}

\section{Monoidal Koszul duality}\label{s:mkd}
In this section, we explain a monoidal upgrade of Theorem~\ref{thm:koszul-regular-rep}.

\subsection{Motivation and statements}\label{ss:mkd-statement}
Let $(\fh,W)$ be a realization. The Hecke category $\cH(\fh^*,W)$ acts on the right quotient $\Kb\ucH(\fh^*,W)$ on the left, and also has a natural ``forgetful'' functor to $\Kb\ucH(\fh^*,W)$: the quotient functor to $\ucH(\fh^*,W)$ composed with the full embedding into $\Kb\ucH(\fh^*,W)$.

We seek a monoidal category that plays the Koszul dual role, filling the top left corner of the following diagram:
\[
 \begin{tikzcd}[column sep=huge, row sep=huge]
  ? \ar[r, dashed, leftrightarrow, "\kappamon", "\sim"']
  \ar[d, dashed, shift left=1.5ex, "\For"] &
  (\cH(\fh^*,W), \star)
  \ar[d, shift left=1.5ex, "\For"] \\
  \Kb\ocH(\fh,W)
  \ar[r, leftrightarrow, "\kappa", "\sim"']
  \ar[loop, dashed, out=100, in=80, distance=50] &
  \Kb\ucH(\fh^*,W).
  \ar[loop, out=100, in=80, distance=50]
 \end{tikzcd}
\]
Here, $\star$ denotes the monoidal product on $\cH(\fh^*,W)$, and $\kappa$ is the equivalence of Theorem~\ref{thm:koszul-regular-rep}. More precisely, we seek a $\bk$-linear additive monoidal category defined in terms of the realization $(\fh,W)$, but which is canonically monoidally equivalent to the dual Hecke category $\cH(\fh^*,W)$. Moreover, this category should admit a natural forgetful functor to $\Kb\ocH(\fh,W)$, and act on the left of the same category, making $\kappa$ into an equivalence of module categories. Said another way, just as the left multiplication action on $\Kb\ucH(\fh^*,W)$ comes from the endomorphism ring of the monoidal identity $B_\id \in \cH(\fh^*,W)$, the desired category provides a monoidal upgrade of the left monodromy action on $\Kb\ocH(\fh,W)$.

The main result of \cite{AMRWa} is the construction of this monoidal category. We will describe its objects (but not the monoidal structure) in \S\ref{ss:FM}. For now, let us state the results. Given a realization $(\fh,W)$, one first defines a $\bk$-linear category
\[
 \FM(\fh,W)
\]
of \emph{free-monodromic complexes}, which should be viewed as the Koszul dual of\\ $\Kb\cH(\fh^*,W)$. Like a left-monodromic complex, a free-monodromic complex is a $\cH(\fh,W)$-sequence equipped with an enhanced ``differential.'' There are two grading shifts $(1)$ and $[1]$ on $\FM(\fh,W)$ and a natural forgetful functor to $\LM(\fh,W) \simto \Kb\ocH(\fh,W)$ intertwining both shifts. The desired monoidal Koszul dual of $\cH(\fh^*,W)$ is a certain full additive subcategory
\[
 (\TiltFM(\fh,W), \hatstar) \subset \FM(\fh,W)
\]
of \emph{free-monodromic tilting sheaves}, equipped with \emph{free-monodromic convolution} $\hatstar$ and stable under $\la1\ra = [1](-1)$. In particular, $\TiltFM(\fh,W)$ contains lifts $\Tmon_w$ of the indecomposable tilting complexes $T_w$ in $\LM(\fh,W) \simto \Kb\ocH(\fh,W)$.

In \cite{AMRWa}, the category $\FM(\fh,W)$ and the operation $\hatstar$ are defined for an arbitrary realization $(\fh,W)$ in terms of the Elias--Williamson diagrammatic Hecke category. However, the proof that $\hatstar$ is bifunctorial uses geometry, so that one needs to restrict to Cartan realizations to define the monoidal category $(\TiltFM(\fh,W), \hatstar)$. We can now state monoidal Koszul duality.
\begin{thm}[Monoidal Koszul duality for Kac--Moody groups \cite{AMRWb}] \label{thm:monoidal-koszul-duality}
 Let $(\fh,W)$ be a Cartan realization over a field $\bk$ of characteristic not equal to $2$. There is a monoidal equivalence
 \[
  \kappamon : (\cH(\fh^*,W), \star) \simto (\TiltFM(\fh,W), \hatstar)
 \]
 satisfying $\kappamon \circ (1) \cong \la1\ra \circ \kappamon$ and sending $\kappamon(B_w) \cong \Tmon_w$ for all $w \in W$. Moreover, the diagram
 \[
  \begin{tikzcd}[row sep=huge, column sep=huge]
   (\TiltFM(\fh,W), \hatstar) \ar[r, leftrightarrow, "\kappa^{\mathrm{mon}}", "\sim"']
   \ar[d, shift left=1.5ex, "\For"] &
    (\cH(\fh^*,W), \star)
   \ar[d, shift left=1.5ex, "\For"] \\
   \Kb\ocH(\fh,W)
   \ar[r, leftrightarrow, "\kappa", "\sim"']
   \ar[loop, out=100, in=80, distance=50] &
   \Kb\ucH(\fh^*,W)
   \ar[loop, out=100, in=80, distance=50]
  \end{tikzcd}
 \]
 commutes up to natural isomorphism, and makes the equivalence $\kappa$ of Theorem~\ref{thm:koszul-regular-rep} an equivalence of module categories compatible with $\kappamon$.
\end{thm}
\begin{rmk}
 The functor $\kappamon$ is defined by generators and relations. That is, viewing $\cH(\fh^*,W)$ as the Elias--Williamson diagrammatic category, $\kappamon$ is defined by specifying the images of its generating objects and morphisms, then checking that they satisfy the defining relations of $\cH(\fh^*,W)$. (This is the categorical analogue of defining a homomorphism out of an algebra defined by generators and relations.) Thus, although Theorem~\ref{thm:monoidal-koszul-duality} may be stated in terms of parity complexes on Kac--Moody flag varieties, the Elias--Williamson monoidal presentation is crucial for the proof.
\end{rmk}

\subsection{Free-monodromic complexes}\label{ss:FM}

Since $\Kb\cH(\fh^*,W)$ has a multiplication action on both left and right, $\FM(\fh,W)$ should have a monodromy action on both left and right. Unlike the left monodromy action described in \S\ref{ss:LM}, the right monodromy action is forced on $\FM(\fh,W)$ in the following way.

Recall that $\Lambda = \Lambda(V^*)$ and $R^\vee = \Sym(V)$. Given a $\cH(\fh,W)$-sequence $\cF$ as in \S\ref{ss:LM}, now consider the further enhancement
\[
 \uEnd_\FM(\cF) := \Lambda \otimes \uEnd(\cF) \otimes R^\vee  \quad (= A \otimes_R \uEnd(\cF) \otimes R^\vee)
\]
of its endomorphism algebra. As with $\uEnd_\LM(\cF)$, $\kappa$ acts via its action on $A$. Consider the canonical element
\begin{equation} \label{eqn:Theta}
 \Theta = \sum (\id_\cF \ustar e_i) \otimes \check e_i \in \uEnd_\FM(\cF),
\end{equation}
where $\{e_i\}$ and $\{\check e_i\}$ are dual bases of $V^* \subset \Lambda$ and $V \subset R^\vee$, and $\ustar$ denotes the multiplication action on $\Kb\cH(\fh,W)$ (or on $\LM(\fh,W) \simto \Kb\ocH(\fh,W)$).
\begin{defn}[{\cite[Definition~5.1.1]{AMRWa}}]
An object of $\FM(\fh,W)$, called a \emph{free-monodromic complex}, is a pair $(\cF,\delta)$, where $\cF$ is a $\cH(\fh,W)$-sequence and $\delta \in \uEnd_\FM(\cF)$. We require that $\delta$ is of an appropriate bidegree, and that it satisfies
\begin{equation} \label{eqn:FM-defn}
 \delta \circ \delta + \kappa(\delta) = \Theta.
\end{equation}
Morphisms in $\FM(\fh,W)$ can also involve both $\Lambda$ and $R^\vee$, and are otherwise defined in much the same way as in $\LM(\fh,W)$ (see Definition~\ref{defn:LM}).
\end{defn}
The rather mysterious condition \eqref{eqn:FM-defn} will be partly explained in Lemma~\ref{lem:appendix}.

The counit $\epsilon_{R^\vee}: R^\vee \to \bk$ induces a map $\epsilon_{R^\vee}: \uEnd_\FM(\cF) \to \uEnd_\LM(\cF)$ that kills $\Theta$ and sends \eqref{eqn:FM-defn} to \eqref{eqn:LM-defn}, hence induces a forgetful functor
\[
 \For: \FM(\fh,W) \to \LM(\fh,W): (\cF, \delta) \mapsto (\cF, \epsilon_{R^\vee}(\delta)),
\]
which should be viewed as the Koszul dual of the natural quotient functor $\Kb\cH(\fh^*,W)$ $\to \Kb\ucH(\fh^*,W)$.

We exhibit two examples of free-monodromic complexes using Soergel bimodules; see \cite[\S5.3]{AMRWa} for more computations involving these examples. Let
\[
 \theta = \sum e_i \otimes \id \otimes \check e_i, \qquad \theta_s = \sum s(e_i) \otimes \id \otimes \check e_i,
\]
where $\{e_i\}$ and $\{\check e_i\}$ are dual bases of $V^*$ and $V$.

\begin{ex} \label{ex:Tmon-id}
The following picture depicts the \emph{free-monodromic unit} $\Tmon_\id$, a free-monodromic complex lifting the left-monodromic complex $T_\id$ from \eqref{eqn:Tid}:
\[
\Tmon_\id := \quad
\begin{tikzcd}
R. \ar[loop right, distance=30, "\theta"]
\end{tikzcd}
\]
In other words, its underlying sequence of Soergel bimodules is $R$ in position $0$, and $\delta_{\Tmon_\id}$ consists of a single nonzero component $R \leadsto R: \theta$. This object is the monoidal unit for free-monodromic convolution.
\end{ex}

\begin{ex} \label{ex:Tmon-s}
Consider again the $\SL_2$ realization $\slreal$ (see Example~\ref{ex:sbim-sl2}). The following picture depicts the \emph{free-monodromic tilting sheaf} $\Tmon_s$, which lifts the left-monodromic complex $T_s$ from Example~\ref{ex:LM-sl2}:
\[
\Tmon_s := \quad
\begin{tikzcd}
R(1) \ar[loop right, in=0, out=20, distance=40, "\theta_s" pos=0.6] \ar[d, bend left=80, "1 \otimes \RBs \otimes \alpha_s^\vee" pos=0.4] \\
B_s \ar[u, "\BsR"] \ar[loop right, in=-20, out=0, distance=50, "\theta_s"] \\
R(-1). \ar[u, "\RBs"] \ar[uu, bend left=60, "-\alpha_s \otimes \id \otimes 1"] \ar[loop right, in=-20, out=20, distance=35, "\theta"]
\end{tikzcd}
\vspace{-10pt}
\]
\end{ex}

We end with an easy lemma that partly explains the condition \eqref{eqn:FM-defn}. Since this lemma does not appear in \cite{AMRWa}, we state it with the precise bigradings, which were not explained in this article.
\begin{lem} \label{lem:appendix}
 Let $(\cF, \delta_\LM)$ be a left-monodromic complex. Choose dual bases $\{e_i\}$ and $\{\check e_i\}$ of $V^*$ and $V$. Choose elements $\delta_i \in \uEnd_\LM(\cF)^1_2$, and set
 \begin{equation} \label{eqn:lem-appendix}
  \delta := \delta_\LM \otimes 1 + \sum \delta_i \otimes \check e_i \in \uEnd_\FM(\cF)^1_0.
 \end{equation}
 Then $(\cF, \delta)$ is a free-monodromic complex if and only if
 \begin{align*}
  d_{\uHom_\LM}(\delta_i) = \id_\cF \ustar e_i \text{ for all } i, \\
  \delta_i \circ \delta_i = 0 \text{ for all } i, \qquad \delta_i \circ \delta_j + \delta_j \circ \delta_i &= 0 \text{ for all } i \neq j,
 \end{align*}
 where $d_{\uHom_\LM}$ is the differential defined in \eqref{eqn:uHom-LM-diff}.
\end{lem}
\begin{proof}
 This follows by comparing the following calculation with \eqref{eqn:FM-defn}:
 \begin{multline*}
  \delta \circ \delta + \kappa(\delta) = \underbrace{\delta_\LM \circ \delta_\LM + \kappa(\delta_\LM)}_{=0} \\
  + \sum_i \underbrace{(\delta_\LM \circ \delta_i + \delta_i \circ \delta_\LM + \kappa(\delta_i))}_{= d_{\uHom_\LM}(\delta_i)} \otimes \check e_i
  + \sum_{i,j} (\delta_i \circ \delta_j) \otimes \check e_i \check e_j. \qedhere
 \end{multline*}
\end{proof}
As in the discussion preceding Example~\ref{ex:trivial-monodromy-lift-counterexample}, one may heuristically think that complexes in $\Kb\ucH(\fh^*,W)$ that admit a lift to $\Kb\cH(\fh^*,W)$ are those with trivial right monodromy. Lemma~\ref{lem:appendix} says that the condition \eqref{eqn:FM-defn} encodes the Koszul dual statement: heuristically, complexes in $\LM(\fh,W)$ that admit a lift to $\FM(\fh,W)$ are those with trivial right multiplication. (As with monodromy, only the forward implication is actually true in general.) Indeed, if $(\cF, \delta)$ is a free-monodromic complex with $\delta$ of the special form in \eqref{eqn:lem-appendix}, then Lemma~\ref{lem:appendix} states that components of $\delta$ with nontrivial $R^\vee$ part encode nullhomotopies for right multiplication on $\For(\cF) \in \LM(\fh,W)$.

\begin{rmk}
 In the recent work of Gorsky--Hogancamp \cite{GH}, an analogous lemma describes the data of a ``strict $y$-ification'' on a bounded complex of $\GL_n$ Soergel bimodules in terms of certain anti-commuting nullhomotopies.
\end{rmk}

\begin{ex}
For the free-monodromic unit $\Tmon_\id$ from Example~\ref{ex:Tmon-id}, each right multiplication $\id_{T_{\id}} \ustar e_i: T_\id \to T_\id(2)$ in $\LM(\fh,W)$ is nullhomotopic with homotopy $h_i = e_i \otimes \id_{T_\id}$. Indeed,
\[
 d_{\uHom_\LM}(h_i) = \delta_{T_\id} \circ h_i + h_i \circ \delta_{T_\id} + \kappa(h_i) = 0 + 0 + e_i \ustar \id_{T_{\id}} = \id_{T_{\id}} \ustar e_i.
\]
\end{ex}
\begin{ex}
For the free-monodromic tilting sheaf $\Tmon_s$ from Example~\ref{ex:Tmon-s}, each right multiplication $\id_{T_s} \ustar e_i: T_s \to T_s(2)$ is nullhomotopic, with homotopy $h_i$ given by
\[
\begin{tikzcd}[column sep=100pt]
R(1) \ar[r, "\sum s(e_i) \otimes \id_R"] \ar[rd, "(\ast)" description]
& R(3) \\
B_s \ar[u, "\BsR"] \ar[r, "\sum s(e_i) \otimes \id_{B_s}"]
& B_s(2) \ar[u] \\
R(-1) \ar[u, "\RBs"] \ar[uu, bend left=60, "-\alpha_s \otimes \id"] \ar[r, "\sum e_i \otimes \id_R"]
& R(1), \ar[u] \ar[uu, bend right=60]
\end{tikzcd}
\]
where $(\ast) = \sum e_i(\alpha_s^\vee) \RBs$. (Note that $\RBs \otimes \alpha_s^\vee = \sum e_i(\alpha_s^\vee) \RBs \otimes \check e_i$.) For example, the component $B_s \leadsto B_s(2)$ of $d_{\uHom_\LM}(h_i)$ equals
\begin{multline*}
 (\sum e_i(\alpha_s^\vee) \RBs) \circ \BsR + \kappa(\sum s(e_i) \otimes \id_{B_s}) \\
 = \sum e_i(\alpha_s^\vee) \BsRBs + s(e_i) \star \id_{B_s} = \id_{B_s} \star e_i.
\end{multline*}
Here, $\BsRBs := \RBs \circ \BsR$, and the last equality uses the so-called polynomial forcing relation in the Elias--Williamson diagrammatic category, which can be checked directly for Soergel bimodules.
\end{ex}

\bibliographystyle{alpha}
\bibliography{refs.bib}

\begin{thebibliography}{AMRW19}

\bibitem[AMRW]{AMRWa}
Pramod~N. Achar, Shotaro Makisumi, Simon Riche, and Geordie Williamson.
\newblock Free-monodromic mixed tilting sheaves on flag varieties.
\newblock preprint ar{X}iv:1703.05843.

\bibitem[AMRW19]{AMRWb}
Pramod~N. Achar, Shotaro Makisumi, Simon Riche, and Geordie Williamson.
\newblock {K}oszul duality for {K}ac--{M}oody groups and characters of tilting
  modules.
\newblock {\em J. Amer. Math. Soc.}, 32:261--310, 2019.

\bibitem[AR]{AR:formal-koszul}
Pramod~N. Achar and Simon Riche.
\newblock Dualit\'e de {K}oszul formelle et th\'eorie des repr\'esentations des
  groupes alg\'ebriques r\'eductifs en caract\'eristique positive.
\newblock preprint ar{X}iv:1703.05843.

\bibitem[AR13]{AR:frobenius}
Pramod~N. Achar and Simon Riche.
\newblock Koszul duality and semisimplicity of {F}robenius.
\newblock {\em Ann. Inst. Fourier (Grenoble)}, 63(4):1511--1612, 2013.

\bibitem[AR16a]{AR:mixed-derived}
Pramod~N. Achar and Simon Riche.
\newblock Modular perverse sheaves on flag varieties, {II}: {K}oszul duality
  and formality.
\newblock {\em Duke Math. J.}, 165(1):161--215, 2016.

\bibitem[AR16b]{ARider}
Pramod~N. Achar and Laura Rider.
\newblock The affine {G}rassmannian and the {S}pringer resolution in positive
  characteristic.
\newblock {\em Compos. Math.}, 152(12):2627--2677, 2016.

\bibitem[AR18]{AR}
Pramod~N. Achar and Simon Riche.
\newblock Reductive groups, the loop {G}rassmannian, and the {S}pringer
  resolution.
\newblock {\em Invent. Math.}, 214(1):289--436, 2018.

\bibitem[ARV]{ARV}
Pramod~N. Achar, Simon Riche, and Cristian Vay.
\newblock Mixed perverse sheaves on flag varieties of {C}oxeter groups.
\newblock preprint ar{X}iv:1802.07651, to appear in Canad. J. Math.

\bibitem[BGG78]{BGG}
I.~N. Bern{\v s}te{\u\i}n, I.~M. Gel{\cprime}fand, and S.~I. Gel{\cprime}fand.
\newblock Algebraic vector bundles on {${\bf P}^{n}$} and problems of linear
  algebra.
\newblock {\em Funktsional. Anal. i Prilozhen.}, 12(3):66--67, 1978.

\bibitem[BGS96]{BGS}
Alexander Beilinson, Victor Ginzburg, and Wolfgang Soergel.
\newblock Koszul duality patterns in representation theory.
\newblock {\em J. Amer. Math. Soc.}, 9(2):473--527, 1996.

\bibitem[BY13]{BY}
Roman Bezrukavnikov and Zhiwei Yun.
\newblock On {K}oszul duality for {K}ac-{M}oody groups.
\newblock {\em Represent. Theory}, 17:1--98, 2013.

\bibitem[EK10]{EK}
Ben Elias and Mikhail Khovanov.
\newblock Diagrammatics for {S}oergel categories.
\newblock {\em Int. J. Math. Math. Sci.}, pages Art. ID 978635, 58, 2010.

\bibitem[Eli16]{Eli}
Ben Elias.
\newblock The two-color {S}oergel calculus.
\newblock {\em Compos. Math.}, 152(2):327--398, 2016.

\bibitem[EW14]{EW:hodge}
Ben Elias and Geordie Williamson.
\newblock The {H}odge theory of {S}oergel bimodules.
\newblock {\em Ann. of Math. (2)}, 180(3):1089--1136, 2014.

\bibitem[EW16]{EW-soergel-calculus}
Ben Elias and Geordie Williamson.
\newblock Soergel calculus.
\newblock {\em Represent. Theory}, 20:295--374, 2016.

\bibitem[Fie08a]{Fie2}
Peter Fiebig.
\newblock The combinatorics of {C}oxeter categories.
\newblock {\em Trans. Amer. Math. Soc.}, 360(8):4211--4233, 2008.

\bibitem[Fie08b]{Fie1}
Peter Fiebig.
\newblock Sheaves on moment graphs and a localization of {V}erma flags.
\newblock {\em Adv. Math.}, 217(2):683--712, 2008.

\bibitem[GH]{GH}
Eugene Gorsky and Matthew Hogancamp.
\newblock Hilbert schemes and $y$-ification of {K}hovanov-{R}ozansky homology.
\newblock preprint ar{X}iv:1712.03938.

\bibitem[JMW14]{JMW}
Daniel Juteau, Carl Mautner, and Geordie Williamson.
\newblock Parity sheaves.
\newblock {\em J. Amer. Math. Soc.}, 27(4):1169--1212, 2014.

\bibitem[KL79]{KL}
David Kazhdan and George Lusztig.
\newblock Representations of {C}oxeter groups and {H}ecke algebras.
\newblock {\em Invent. Math.}, 53(2):165--184, 1979.

\bibitem[KL80]{KL2}
David Kazhdan and George Lusztig.
\newblock Schubert varieties and {P}oincar\'e duality.
\newblock In {\em Geometry of the {L}aplace operator ({P}roc. {S}ympos. {P}ure
  {M}ath., {U}niv. {H}awaii, {H}onolulu, {H}awaii, 1979)}, Proc. Sympos. Pure
  Math., XXXVI, pages 185--203. Amer. Math. Soc., Providence, R.I., 1980.

\bibitem[Maka]{Maka}
Shotaro Makisumi.
\newblock Mixed modular perverse sheaves on moment graphs.
\newblock preprint arXiv:1703.01571.

\bibitem[Makb]{Makb}
Shotaro Makisumi.
\newblock Modular {K}oszul duality for {S}oergel bimodules.
\newblock preprint arXiv:1703.01576.

\bibitem[MR18]{MR}
Carl Mautner and Simon Riche.
\newblock Exotic tilting sheaves, parity sheaves on affine {G}rassmannians, and
  the {M}irkovi\'c--{V}ilonen conjecture.
\newblock {\em J. Eur. Math. Soc. (JEMS)}, 20(9):2259--2332, 2018.

\bibitem[Ric16]{Ric:habilitation}
Simon Riche.
\newblock {G}eometric {R}epresentation {T}heory in positive characteristic.
\newblock 2016.
\newblock M\'{e}moire d'habilitation, Universit\'{e} Blaise Pascal (Clermont
  Ferrand 2), available from
  \url{https://tel.archives-ouvertes.fr/tel-01431526}.

\bibitem[RW18]{RW}
Simon Riche and Geordie Williamson.
\newblock Tilting modules and the {$p$}-canonical basis.
\newblock {\em Ast\'erisque}, (397):ix+184, 2018.

\bibitem[Soe97]{Soe}
Wolfgang Soergel.
\newblock Kazhdan-{L}usztig polynomials and a combinatoric[s] for tilting
  modules.
\newblock {\em Represent. Theory}, 1:83--114, 1997.

\bibitem[Soe07]{Soe:bimod}
Wolfgang Soergel.
\newblock Kazhdan-{L}usztig-{P}olynome und unzerlegbare {B}imoduln \"uber
  {P}olynomringen.
\newblock {\em J. Inst. Math. Jussieu}, 6(3):501--525, 2007.

\bibitem[Wil]{Wil:ICM}
Geordie Williamson.
\newblock Parity sheaves and the {H}ecke category.
\newblock Proceedings of the International Congress of Mathematicians 2018 (ICM
  2018).

\bibitem[Wil17]{Wil:Takagi}
Geordie Williamson.
\newblock Algebraic representations and constructible sheaves.
\newblock {\em Jpn. J. Math.}, 12(2):211--259, 2017.

\end{thebibliography}

\end{document}